\documentclass[11pt]{amsart}

\usepackage{amssymb}
\usepackage[all]{xy}
\usepackage[colorlinks=true, urlcolor=rltblue, citecolor=drkgreen, linkcolor=drkred] {hyperref}
\usepackage{color}
\usepackage{pdfsync}
\definecolor{rltblue}{rgb}{0,0,0.4}
\definecolor{drkred}{rgb}{0.6,0,0}
\definecolor{drkgreen}{rgb}{0,0.4,0}
\usepackage{graphicx}
\usepackage[mathscr]{eucal}

\usepackage{lmodern}
\usepackage[capitalize]{cleveref}
\usepackage{amssymb,amsmath,amsthm}
\usepackage{mathtools, MnSymbol}
\usepackage{setspace}
\usepackage{tikz-cd}
\usepackage[T1]{fontenc}
\usepackage[utf8]{inputenc}
\usepackage{textcomp} 
\usepackage{stmaryrd}
\usepackage{datetime}
\usepackage{todonotes}
\usepackage{xcolor}
\usepackage{mdframed}
\usepackage[all]{xy}

\usepackage{thmtools}
\usepackage{enumitem}
\declaretheorem{theorem}
\declaretheorem[sibling=theorem]{lemma}

\declaretheorem[sibling=theorem]{corollary}
\declaretheorem[sibling=theorem]{definition}

\newcommand{\A}{\mathcal{A}}
\newcommand{\B}{\mathcal{B}}
\newcommand{\C}{\mathcal{C}}

\newcommand{\SR}{\text{SR}}

\renewcommand{\L}{\mathcal{L}}

\renewcommand{\phi}{\varphi}

\newmdtheoremenv[backgroundcolor=cyan]{theorem-prove}{Theorem}[theorem]
\newmdtheoremenv[backgroundcolor=cyan]{lemma-prove}{Lemma}[theorem]
\newmdtheoremenv[backgroundcolor=cyan]{proposition-prove}{Proposition}[theorem]

\newmdtheoremenv[backgroundcolor=yellow!40]{theorem-check}{Theorem}[theorem]
\newmdtheoremenv[backgroundcolor=yellow!40]{lemma-check}{Lemma}[theorem]
\newmdtheoremenv[backgroundcolor=yellow!40]{proposition-check}{Proposition}[theorem]

\def\abar{{\bar{a}}}
\def\bbar{{\bar{b}}}

\def\hbar{{\bar{h}}}

\def\mbar{{\bar{m}}}
\def\nbar{{\bar{n}}}

\def\xbar{{\bar{x}}}

\def\lom{{<\omega}}
\def\a{\alpha}
\def\b{\beta}
\def\g{\gamma}

\def\om{\omega}
\def\si{\sigma}
\def\Si{\Sigma}

\def\A{\mathcal A}
\def\B{\mathcal{B}}
\def\C{\mathcal{C}}
\def\D{{\mathcal D}}

\def\L{\mathcal L}

\def\itt{{\mathtt {in}}}

\def\Sii{\Si^\itt}
\def\Pii{\Pi^\itt}

\def\Lomom{\L_{\om_1,\om}}

\newtheorem{thm}{Theorem}
\newtheorem*{lemma_NN}{Lemma}

\theoremstyle{remark}

\newtheorem{obs}[thm]{Observation}

\def\upto{\mathbin{\upharpoonright}}

\def\and{\mathrel{\&}}

\def\isom{\cong}

\def\vo{vo}
\def\BA{\mathtt{BA}}

\title{The $\om$-Vaught's Conjecture}

\author{David Gonzalez}
\address{Department of Mathematics\\
University of California, Berkeley}
\email{david\_gonzalez@berkeley.edu}   

\author{Antonio Montalb\'an}
\thanks{The second author was partially supported by NSF grant DMS-1363310.}

\address{Department of Mathematics\\
University of California, Berkeley}

\email{antonio@math.berkeley.edu}
\urladdr{\href{http://www.math.berkeley.edu/~antonio/index.html}{www.math.berkeley.edu/$\sim$antonio}}

\date{Compiled: \today, \currenttime}

\begin{document}
\maketitle

\begin{abstract}
We introduce the $\om$-Vaught's conjecture, a strengthening of the infinitary Vaught's conjecture. We believe that if one were to prove the infinitary Vaught's conjecture in a structural way without using techniques from higher recursion theory, then the proof would probably be a proof of the $\om$-Vaught's conjecture. We show the existence of an equivalent condition to the $\om$-Vaught's conjecture and use this tool to show that all infinitary sentences whose models are linear orders satisfy the $\om$-Vaught's conjecture. 
\end{abstract}

Robert Vaught conjectured in \cite{Vau61} that the number of countable models of any given list of axioms\footnote{on a countable language} must be either countable or continuum, but never in between. 
Despite all the work that has gone into this conjecture over the past sixty years, it remains open.
It is one of the most well-known, long-standing open questions in mathematical logic.
In this paper we will consider the {\em infinitary Vaught's conjecture}, where the list of axioms can be taken to be an infinitary sentence from $\L_{\om_1,\om}$.
An interesting aspect of Vaught's conjecture is that it connects many areas of logic. 
It is unclear where the answer is going to come from.
The original version was for finitary first-order theories.
If a solution of the finitary version comes first, it will probably come from model theory.
The infinitary version, though, has been proved to be equivalent to statements in computability theory (see \cite{MonVC, MonIntermediate}) and descriptive set theory (see \cite[Theorem 11.3.8]{GaoBook}.)

John Steel \cite{Ste78} proved the infinitary Vaught's conjecture for all  theories in the language of orderings that imply the axioms of linear orderings. 
Rubin \cite{Rub74} had already proved it for finitary extensions of linear orderings. 
Gao \cite{Gao01} modified Steel's proof to show the extensions of the theory of linear ordering satisfy the Glimm--Effros dichotomy, providing a descriptive set theoretic reason for why linear orderings satisfy the infinitary Vaught's conjecture. 
Montalb\'an \cite{MonIntermediate} also modified Steel's proof to study the isomorphism relation on linear orderings from a computability theoretic perspective and showed they satisfy the {\em no-intermediate extension property}.

Steel, Gao, and Montalb\'an used higher recursion theoretic techniques, such as $\Sigma^1_1$-bounding and considered models of Scott ranks up to $\om_1^{CK}$ (or actually, $\om_1^{T}$, the first ordinal not computable in the tree representation of the sentence $T$). 
It is rather surprising that a recursion theoretic lemma like $\Si^1_1$-bounding and a recursion theoretic notion like $\om_1^T$ would have to do with the number of countable models of $T$.
In this paper we give a more structural proof of the infinitary Vaught's conjecture for linear orderings.

We also propose a strengthening of Vaught's conjecture that we call $\om$-VC. 
We believe that if one were to prove the infinitary Vaught's conjecture in a structural way without using techniques from higher recursion theory, then the proof would probably be a proof of  $\om$-VC.
Also, we expect that if a counterexample to $\om$-VC were to be found, it will probably contain ideas that could be used to build a counterexample to Vaught's conjecture.

Gerald Sacks \cite[Section 5]{Sac07} defined the {\em Vaught rank} of a theory $T$ as an ordinal that, in a sense, witnesses that $T$ satisfies Vaught's conjecture. 
His definition is rather complicated to describe, and we omit it here.
In a similar vein, we define the {\em Vaught ordinal} of a theory $T$, which we denote as $\vo(T)$,  as least ordinal $\b$ such that 
\begin{itemize}
\item either there are only countably many models of $T$ and they all have Scott rank less than $\b$,
\item or there are uncountably many models of $T$ which are not $\Pii_\b$-elementary equivalent with each other.
\end{itemize}
The ordinal $\vo(T)$ tells us how high we need to go in the $\Pii_\b$ hierarchy of infinitary sentences to make sure that $T$ satisfies Vaught's conjecture.
We will define the Scott rank of a structure in definition \ref{def:ST} below. For now, let us say that for a limit ordinal $\lambda$, a countable structure $\A$ {\em has Scott rank $\geq \lambda$} if for all $\g<\lambda$, there is another countable structure $\B$ that is $\equiv_\g$-equivalent to $\A$ but not isomorphic to $\A$. 
We use the notation $\A\equiv_\b\B$ to say that the structures $\A$ and $\B$ are $\Pii_\b$-elementary equivalent, i.e., that they satisfy the same $\Pii_\b$-sentences. 
It is known that the equivalence relations  $\equiv_\b$ are Borel and that they approximate the isomorphism equivalence relation in the sense that, for countable structures $\A$ and $\B$,  $\A\isom\B\iff (\forall \b<\om_1)\ \A \equiv_\b\B$.
Silver's theorem states that a Borel equivalence relation must have either countably many or continuum many equivalence classes. 
One can then show that, for a counterexample of Vaught conjecture, there would be no upper bound for the value of countable $\b$ such that $\A\equiv_\b \B$ yet $\A\not\cong \B$ among models of $T$, and that for each $\b<\om_1$, there would be only countably many many models of $T$ up to $\equiv_\b$-equivalence. 
It follows that Vaught's conjecture holds if and only if $\vo(T)<\om_1$ for all $\Lomom$ sentences $T$.
It follows from Steel's results \cite{Ste78} that, if $T$ is a sentence extending the theory of linear orderings, then $\vo(T)\leq \om_1^T$.
Note that this upper bound, $\om_1^T$, does not just depend on the quantifier complexity of the sentence $T$, but also on the computational complexity of the real that is coding the tree representation of sentence $T$.
This is rather unexpected. 
The main result of this paper gives a much more reasonable upper bound for the Vaught ordinal of $T$, one that depends only on the quantifier complexity of $T$, and in an additive way:

\begin{theorem}
Consider the vocabulary $\tau = \{\leq\}$ of linear orderings. 
For every $\a<\om_1$ and every $\Pii_\a$-$\tau$-sentence extending the axioms of linear orderings,
\[
\vo(T) \leq \a+\om.
\] 
\end{theorem}

This theorem provides a structural understanding of why the infinitary Vaught's conjecture is true for linear orderings. 

\begin{definition}
We say that an $\Lomom$ sentence $S$ satisfies {\em $\om$-VC} if for every $\a<\om_1$ and every $\Pii_\a$ sentence $T$ that implies $S$, we have $\vo(T) \leq \a+\om$.
\end{definition}

The second main result of this paper is the following lemma, which can be seen as a tool to prove that a theory satisfies $\om$-VC.
We need to define a couple of notions first:

\begin{definition}
Given ordinals $\a<\b$, we say that a structure $\A$ is {\em $(\a,\b)$-small} if, for every $\g<\b$, there only countably many $\equiv_\g$-equivalence classes among the structures that are $\equiv_\a$-equivalent to $\A$.
\end{definition}

\begin{lemma}\label{lem: VC splitting}
Let $S$ be an $\Lomom$ sentence.
The following are equivalent:
\begin{enumerate}
\item $S$ satisfies $\om$-VC.			\label{part: VC splitting 1}
\item For every $\a<\om_1$ that is greater than the quantifier rank of $S$%
\footnote{The quantifier rank of $S$ is the least $\b$ such that $S$ is either $\Pii_\b$ or $\Sii_\b$.}
 and every model $\A$ of $S$ that is $(\a,\a+\om)$-small and has Scott rank $\geq \a+\om$, there is another model $\B$ of $S$ of Scott rank $\geq \a+\om$ for which we have $\A\equiv_\a\B$ but $\A\not\equiv_{\a+\om}\B$.		\label{part: VC splitting 2}
\end{enumerate}
\end{lemma}

Before moving on, let as make a few quick observations about $\om$-VC.
First, we note if one wanted to prove that all sentences satisfy $\om$-VC, it would be enough to consider only $\Pii_2$ theories:

\begin{obs}
The following are equivalent:
\begin{itemize}
\item Every $\Lomom$ sentence satisfies  $\om$-VC.
\item For every $\Pii_2$ sentence $T$, $\vo(T)\leq \om$.
\end{itemize}
\end{obs}
This observation is easily verified by taking Morleyizations (see for instance \cite[Chapter II.5]{Part2}).

Let us mention the following interesting example:
The theory, $\BA$, of Boolean algebras, is $\Pii_2$ and has Vaught ordinal $\om$.
This means that the condition that $\vo(T)\leq \om$ in the observation above cannot be strengthened. 
It is not known if the extensions of $\BA$ satisfy Vaught's conjecture or $\om$-VC.
In unpublished work, Montalb\'an and Simon \cite{MSunpub} showed that if there is an extension of $\BA$ that does not satisfy Vaught's conjecture, then there is one that, for some ordinal $\a$ and Boolean algebras $\B_i$ of Scott rank $\leq \a$, it says that its models are $\equiv_{\a+\om}$-equivalent to the $\om$-sum $\bigoplus_{i\in\om}\B_i$.
They also proved that for every $\b$ there exists Boolean algebras $\B_i$ of Scott rank $<\om+\om$ such that $\bigoplus_{i\in\om}\B_i$ has Scott rank greater than $\b$.

Let us also remark that $\om$-VC plays a similar role as the {\em Martin's model-theoretic conjecture}, which is about complete {\bf finitary} first-order theories with less than continuum many countable models.
It plays a similar role in the sense that it is suggested that if Vaught's conjecture were to be proved by model theoretic means, then it would probably be through Martin's model-theoretic conjecture.
It is also somewhat similar in the sense that it implies that if a complete finitary first-order theories has less than continuum many countable models, then all those models have Scott rank less than or equal to a certain bound, that for Martin's conjecture is $\om+\om$.
However, Martin's conjecture and $\om$-VC are incomparable, and neither implies the other one as far as we know.
Wagner \cite{Wag82} had proposed a strengthening of Martin's conjecture which included theories with continuum many countable models, which turned out to be false
(see \cite{Gaob01}).


\section{Preliminaries on Scott Rank and Notation}

In this section we will explain the notation and basic concepts used in this paper. 
We refer the reader to \cite[Chapter 2]{Part2} for more background on the concepts defined here.

The logic used in this paper is $L_{\omega_1,\omega}$. 
For a structure $\mathcal{M}$, a tuple $\mbar$ from $M$, and an ordinal $\alpha$, the {\em $\Pii_\alpha$-type of $\mbar$ in $\mathcal{M}$} is the set of true $\Pii_\alpha$ formulas about $\mbar$ in $\mathcal{M}$ (similar for $\Sii_\alpha$). We write $(\mathcal{M},\mbar)\leq_\alpha (\mathcal{N},\nbar)$ if the $\Pii_\alpha$-type of $\mbar$ in $\mathcal{M}$ is contained in the $\Pii_\alpha$ type of $\nbar$ in $\mathcal{N}$. This is called the $\alpha$ \textit{back and forth relation}. 
We then have that $(\mathcal{M},\mbar)\equiv_\alpha (\mbar,\mathcal{N})$ if and only if $\mathcal{M}\leq_\alpha \mathcal{N}$ and $\mathcal{N}\leq_\alpha \mathcal{M}$.

One of the primary motivations for studying $L_{\omega_1,\omega}$ is that it has enough power to describe automorphism orbits within a structure. This leads to the following definition.

\begin{definition}\label{def:ST}
The \textit{Scott rank} of a structure $\mathcal{M}$ is given by the least $\alpha\in\omega_1$ for which there exists a finite tuple of parameters $\mbar$ such that the automorphism orbit of every tuple, $\xbar\in\mathcal{M}$, is $\Sii_\alpha$ definable  over $\mbar$. 
We denote this $\SR(\mathcal{M})$.
\end{definition}

The definition above is taken from \cite{MonSR} and \cite{Part2}. In recent work, it is sometimes called the parameterized Scott rank and is contrasted with an unparameterized version that does not allow for a finite tuple of parameters to be used in the definition of the automorphism orbits. For our (parameterized) notion of Scott rank, we have the following equivalence. 

\begin{theorem}
\cite{MonSR} $\mathcal{M}$ has Scott Rank $\alpha$ if and only if there is a $\Sii_{\alpha+2}$ sentence that is true about $\mathcal{M}$ and not true about any other countable structure.
\end{theorem}

We call such a sentence a $\textit{Scott Sentence}$ for $\mathcal{M}$. Note that for a fixed $\mathcal{M}$ with $\SR(\mathcal{M})=\alpha$, this result gives that statements of the form $\mathcal{N}\cong\mathcal{M}$ are $\Sii_{\alpha+2}$.

Of the many benefits of this notion of Scott rank, and the one most useful in this paper is that the Scott rank of a structure can be seen by a formula of moderate complexity from inside the structure (see \cite[Lemma II.67]{Part2}). 

\begin{lemma}\label{intsr}
For a fixed vocabulary, given any ordinal $\alpha,$ there is a $\Pii_{2\alpha +3}$ sentence $\rho_\alpha$ such that
$$\mathcal{A}\models \rho_\alpha \iff \SR(\mathcal{A})\geq \alpha.$$
\end{lemma}

Note that this immediately gives that stating $\SR(\mathcal{A})=\alpha$, i.e. $\mathcal{A}\models \rho_\alpha\land\lnot \rho_{\alpha+1}$ is $\Sigma_{2\alpha +4}$. A useful consequence of this is that if $\alpha=\lambda+n$ where $\lambda$ is a limit ordinal and $n\in\omega$, this statement is $\Sigma_{\lambda + 2n +4}$. In other words, the complexity of the sentence stops well short of the first limit ordinal strictly larger than $\alpha$.

It is also true that the back and forth relations among structures can be defined from inside the structure. 
(See \cite[Lemma VI.14]{Part2}.) 

\begin{lemma}\label{bound}
Let $\mathcal{L}$  be a structure. For any tuple $\abar\in L$ and $\beta<\gamma$ there are  $\Pii_{2\beta}$ formulas $\varphi_{\abar,\beta}(x)$ and   $\psi_{\abar,\beta}(x)$ such that for any $\mathcal{K}$ and tuple $\bbar\in\mathcal{K}$
\[
\mathcal{K}\models \varphi_{\abar,\beta}(\bbar) \iff (\mathcal{L},\abar)\leq_\beta(\mathcal{K},\bbar),
\]
and
\[
\mathcal{K}\models \psi_{\abar,\beta}(\bbar) \iff (\mathcal{L},\abar)\geq_\beta(\mathcal{K},\bbar).
\]
\end{lemma}

Notice that in the same manner as the previous argument, we can note that the complexity of this formula is well below the least limit ordinal strictly larger than $\beta$.


\section{Proof of the main lemma}

In this section we give a proof of Lemma \ref{lem: VC splitting}, which will be key for our proof that linear orderings satisfy $\om$-VC in the next sections.

\begin{lemma_NN}[Lemma \ref{lem: VC splitting}]
Let $S$ be an $\Lomom$ sentence.
The following are equivalent:
\begin{enumerate}
\item $S$ satisfies $\om$-VC.			
\item For every $\a<\om_1$ that is greater than the quantifier rank of $S$ and every model $\A$ of $S$ that is $(\a,\a+\om)$-small and has Scott rank $\geq \a+\om$, there is another model $\B$ of $S$ of Scott rank $\geq \a+\om$ for which we have $\A\equiv_\a\B$ but $\A\not\equiv_{\a+\om}\B$.		
\end{enumerate}
\end{lemma_NN}

\begin{proof}
To see that (\ref{part: VC splitting 1}) implies (\ref{part: VC splitting 2}), consider a model $\A$ of $S$ that is $(\a,\a+\om)$-small and has Scott rank $\geq \a+\om$.
Let $T$ be the $\Pii_{2\a}$-sentence from Lemma \ref{bound} satisfying that $\B\models T$ if and only if $\A\equiv_\a \B$.
Since $S$ satisfies $\om$-VC and $\A$ already has rank $\geq\a+\om$, we must then have that $T$ has continuum many models up to $(\a+\om)$-elementary equivalence.
To obtain (\ref{part: VC splitting 2}), we need to show that at least two of these models has Scott rank $\geq \a+\om$.
If not, there would be some $\g<\a+\om$ such that for continuum many of these $(\a+\om)$-elementary equivalence clasess, there would be models of Scott rank $\g$.
But then, we would have continuum many $\g$-equivalence classes among the models of $T$, contradicting that $\A$ is $(\a,\a+\om)$-small.

Suppose now that (\ref{part: VC splitting 2}) holds, let $\a$ be an ordinal, and let $T$ be a $\Pii_\a$ extension of $S$.
We need to show that $\vo(T)\leq \a+\om$.
If all models of $T$ have  Scott rank $<\a+\om$, then $T$ has Vaught ordinal $\leq\a+\om$ as wanted. 
So, let us suppose that $T$ that has a model $\A$ of Scott rank $\geq\a+\om$.
If $\A$ were not $(\a,\a+\om)$-small, we would immediately have a $\g<\a+\om$ for which there are continuum many models of $T$ up to $\equiv_\g$-equivalence and $T$ would have Vaught ordinal $\leq\a+\om$ as wanted. 
So suppose $\A$ is $(\a,\a+\om)$-small.

We will build a continuum many models $\{\B_X: X\in 2^{\om}\}$ of $T$ that are not $(\a+\om)$-elementary equivalent.
For that, we first build a tree $\{\B_\si: \si\in 2^{\lom}\}$ of models of $T$ of Scott rank $\geq\a+\om$ and an increasing sequence of natural numbers $\{n_i: i\in\om\}$ such that
\begin{itemize}
\item If $\si$ and $\tau$ are incompatible strings of length $i$, then $\B_{\si}\not\equiv_{\a+n_i} \B_\tau$.
\item If $\si\subseteq \tau$ with $|\si|=i$, then $\B_{\si}\equiv_{\a+n_i+3} \B_\tau$.
\end{itemize}
 
Let $\B_\emptyset$ be $\A$ and let $n_0=0$.
Suppose we have defined $\B_\si$ for all $\si\in 2^\om$ of length $i$.
Fix such a $\si$, and let us define $\B_{\si 0}$ and $\B_{\si 1}$.
Let $\B_{\si 0}=\B_{\si}$.
Since $\B_\si\equiv_\a \A$,  $\B_\si$ must also be $(\a,\a+\om)$-small, and in particular $(\a+n_i,\a+\om)$-small.
By (\ref{part: VC splitting 2}), there is a model $\B_{\si 1}$ of $T$ of Scott rank $\geq \a+\om$ for which we have $\B_{\si 0}\equiv_{\a+n_i+3}\B_{\si 1}$ but $\B_{\si 0}\not\equiv_{\a+\om}\B_{\si 1}$.
It is known that for limit ordinals $\lambda$, $\equiv_\lambda$ is the limit of $\equiv_\g$ for $\g<\lambda$; that is, if $\C\not\equiv_\lambda\D$, then there is some $\g<\lambda$ such that $\A\not\equiv_\g\B$ (this follows easily from the back-and-forth definition of $\equiv_\lambda$, see \cite[Definition II.3.2]{Part2}).
Let $n_\si$ be such that $\B_{\si 0}\not\equiv_{\a+n_\si}\B_{\si 1}$.
Finally, let $n_{i+1}$ be the maximum of $n_\si$ for all $\si\in 2^i$.
This finishes the construction of the tree. 

For each $X\in 2^\om$, we have a sequence of structures
\[
B_\emptyset \equiv_{\a+3} B_{X\upto 1} \equiv_{\a+n_1+3}  B_{X\upto 2} \equiv_{\a+n_2+3} B_{X\upto 2} \equiv_{\a+n_3+3}	\cdots 
\]
It is proved in \cite[Lemma XII.6]{Part2} that given such a sequence, there exists a structure $\B_X$ such that $\B_X\equiv_{n_i} \B_{X\upto i}$ for every $i\in\om$.
For different $X, Y\in 2^\om$, let $i$ be such that $X\upto i\neq Y\upto i$. 
Then 
\[
B_X \equiv_{\a+n_{i}} \B_{X\upto i} \not\equiv_{\a+n_i} \B_{Y\upto i} \equiv_{\a+n_{i}} \B_Y.
\]
So we get continuum many models of $T$ up to $\equiv_{\a+\om}$-equivalence, and thus $\vo(T)\leq \a+\om$.
\end{proof}

With this result, we have now established the needed background theory on the robust Scott rank needed for this paper.



\section{Operations on Linear Orders and Scott Rank}

This section will analyze the Scott ranks of various types of linear orders. A similar analysis was done in \cite{MonIntermediate}, and some of the results are directly from that paper. Other results along these lines are improvements of the results in \cite{MonIntermediate} that are needed to obtain the claimed bound.

We use the standard notation $[a,b],(a,b),L_{>a},L_{\geq a},L_{<a},L_{\leq a}$ to speak about open and closed intervals, initial segments and end segments of a linear order $L$. 

The following basic result about countable linear orders is needed many times.

\begin{lemma}\label{cofseq}
Given any countable linear order without a greatest element, $L$, there exists a cofinal, injective, order preserving map $\omega\to L$. 
\end{lemma}

\begin{proof}
Consider an enumeration of the elements $\{l_i\}_{i\in\omega}$. Let $h(0)=l_0$ and $h(j+1)$ to be the first enumerated element that is larger than both $l_j$ and $h(j)$. It is clear that this map is as desired.
\end{proof}

The following few results describe the complexity of linear orders based on the complexity of the suborders that act as various types of building blocks for the final order. The first one was shown in \cite{MonIntermediate}. While it used a slightly different notion of Scott rank, this proof is sufficient for our notion as well.

\begin{lemma}\label{least}
(\cite[Lemma 4.3]{MonIntermediate}) Let $L=A+1+B$ be a sum of linear orders. We have that
$$\max(\SR(A),\SR(B))\leq \SR(L).$$
\end{lemma}

While not explicitly stated in \cite{MonIntermediate}, the same proof also applies to see that $\max\{\SR(A+1),\SR(1+B)\}\leq \SR(L)$. In other words, this result applies any time $L$ is realized as a sum where the bottom summand has a greatest element and/or the top summand has a least element, and these can be taken to be a single overlapping element if needed.

Note that this inequality is simply a case of the general fact that any $N\subseteq M$ that is $\Delta_0^{in}$ definable over parameters in $M$ has $\SR(N)\leq\SR(M)$ (again the proof in \cite{MonIntermediate} is sufficient to see this). We occasionally refer to this more general result as well. 

There is also a corresponding upper bound for this result presented in \cite{MonIntermediate}. We provide a slightly improved version of the upper bound that deals with the robust Scott rank defined above and that will better suit the purposes here. 

\begin{lemma}\label{leasteq}
Given a linear orders $A$ and $B$ we have that 
$$\SR(A+1+B)=\max(\SR(A),\SR(B)).$$
\end{lemma}

\begin{proof}
Let $\gamma = \max(\SR(A),\SR(B))$. Name the element representing the "1" in the given decomposition $c$ and consider the structure $(L,c)$. Note that an automorphism of this structure must preserve the elements above $c$ and those below $c$. In other words, it must factor into an automorphism of $A$ and $B$. Thus, for $b\in B$, its automorphism orbit in $L$ is exactly its automorphism orbit in $B$. In particular, if $\varphi\in\Sigma_\gamma^{in}$ describes the automorphism orbit of $b\in B$, we have that $\varphi^{>c}$ describes the automorphism orbit of $b\in L$, where $\varphi^{>c}$ is the same as $\varphi$ save for the fact that all quantifiers are restricted to $x>c$. Note that $\varphi^{>c}\in\Sigma_\gamma^{in}(c)$. Similar analysis holds for $a\in A$. Overall, this gives that $\SR(L)\leq\gamma.$

The other inequality follows from the result of Montalb\'an explained above.
\end{proof}

This version of the lemma allows a result that bounds the complexity any finite sum of linear orders with minimal elements. 

\begin{corollary}\label{finsum}
If $L$ can be written as a finite sum of the form
$$L=\sum_{i\in n} 1+A_i$$
where $\SR(A_i)\leq\gamma$, we have that $\SR(L)\leq\gamma$.
\end{corollary}

\begin{proof}
Note that $\SR(A_i)\leq\gamma$. By $n$ repeated applications of Lemma \ref{leasteq}, we immediately see that $\SR(L)\leq\gamma$. 
\end{proof}

We also provide an improved result on the Scott rank of sums of linear orders that will ultimately help improve the bounds provided in this paper. The result is most clear if we start with a lemma that concerns automorphism orbits within initial segments of the order. To state this lemma we need the following definition. 

\begin{definition}
Given a linear order $K$ and points $z<b$ in $K$, let $\varphi_{z,b}$ denote a definition of the automorphism orbit of $z$ within $K_{<b}$.
\end{definition}

We are now ready to prove the following lemma, which provides the key technical insight needed to bound the complexity of sums of linear orders.

\begin{lemma}\label{fixaut}
Given a linear order $K$ and a point $z\in K$, if there is a point $b>z$ such that every $y>b$ with $K_{<b}\cong K_{<y}$ has that $\varphi_{z,b}$ also defines the automorphism orbit of $z$ in $K_{<y}$, then the automorphism orbit of $z$ within $K$ can be described by a $\Sigma_{\SR(K_{<b})+4}^{in}$ formula.
\end{lemma}

\begin{proof}
Consider the following formula:
$$\Phi_z(w):\hspace{.5 in} \exists v>w ~ \Big(\big(K_{<v}\cong K_{<b}\big) \land \forall y>v ~ \big(K_{<y}\cong K_{<b} \to K_{<y}\models \varphi_{z,b}(w)\big)\Big).$$

We show that this formula characterizes the automorphism orbit of $z$. 
From this, the lemma will immediately follow as it is of the desired complexity using the fact that $K_{<y}\cong K_{<b}$ is a $\Sigma_{\SR(K_{<b})+2}^{in}$ formula.

Firstly, observe that $\Phi_z(z)$ as $b$ is a witness for this statement for $z$. In particular, note that by definition of $b$, for all $y$ with $K_{<y}\cong K_{<b}$ we have that $\varphi_{z,b}$ describes the automorphism orbit of $b$ in $K_{<y}$. This gives that $\varphi_{z,b}(z)$. Therefore, any element in the automorphism orbit of $z$ also satisfies $\Phi_z$.

Lastly, note that if $\Phi_z(w)$ there is a $y>\max(w,z)$ such that $K_{<y}\models \varphi_{z,b}(w)$. Thus, there is an automorphism of $K_{<y}$ taking $w$ to $z$. This can be extended to an automorphism of $K$ by fixing $K_{\geq y}$. Therefore, $w$ is in the automorphism orbit of $z$ as desired.
\end{proof}

With this in mind, we are nearly ready to bound the Scott rank of the sum of linear orders. That being said, the proof of this lemma will also make use of the following Lemma of Lindenbaum. 

\begin{definition}
Given two linear orders $L,K$ we say that $L\sqsubseteq K$ if $L$ is an initial segment of $K$. 
We say $L\sqsupseteq K$ if $K$ is a final segment of $L$.
\end{definition}

\begin{lemma}\label{lin}
(Lindenbaum, \cite{Ros82}) Given two linear orders $L,K$ if $L\sqsubseteq K$ and $L\sqsupseteq K$ then $L\cong K$.
\end{lemma}

With this in place, we can now prove the lemma we were aiming at. It is a tricky combinatorial argument that breaks the behavior of linear orders into several cases and subcases. Ultimately we will find suitable witnesses to Lemma \ref{fixaut} that enable the simple description of automorphism orbits within a sum.

\begin{lemma}\label{sum}
Given linear orders $A$ and $B$, we have that $\SR(A+B)\leq\max(\SR(A),\SR(B))+4$.
\end{lemma}

\begin{proof}
Let $\gamma=\max(\SR(A),\SR(B))$.

Case 1: There is an $a\in A$ and $b\in B$ such that $A_{<a}\cong(A+B)_{<b}$.

In this case, 
$$\SR((A+B)_{<b})=\SR(A_{<a})\leq\SR(A)\leq\gamma,$$ and 
$$\SR((A+B)_{>b})=\SR(B_{>b})\leq\SR(B)\leq\gamma.$$ As $A+B=(A+B)_{<b}+1+(A+B)_{>b},$ this yields $\SR(A+B)\leq\gamma\leq\gamma+4$ as desired.

Case 2: Otherwise.

We claim that in this case, for some $a\in A$ and $b\in B$ we have that every point $z\in (a,b)$ has an automorphism orbit described by a formula in $\Sigma^{in}_{\gamma+4}$ in $(a,b)$. As $\SR(A_{<a})\leq \gamma$ and $\SR(B_{>b})\leq \gamma$ the result will then follow from Corollary \ref{finsum}. To be more specific, we will demonstrate this claim by appealing to Lemma \ref{fixaut}. Note that if we prove the claim for points $z\in A$, by symmetry (looking at $(A+B)^*$) we obtain the result for points in $B$ as well. For this reason we focus on points in $A$ in the following argument.

Subcase 1: There are cofinally many $x\in A$ such that the set $\{y|A_{<y}\cong A_{<x}\}$ is bounded in $A$.

Given a $z\in A$ we will find a $b\in A$ that satisfies the hypothesis of Lemma \ref{fixaut}. By the case we are in, we can take an $x>z$ with $x\in A$ such that the set $\{y|A_{<y}\cong A_{<x}\}$ is bounded by some $b\in A$. Consider a $v>b$ with $A_{<b}\cong(A+B)_{<v}$. Because we are not in case 1, $v\in A$, so we can say $A_{<b}\cong A_{<v}$. It is apparent that the left part of the cut in $A_{<v}$ defined by the points in $\{y|A_{<y}\cong A_{<x}\}$ is automorphism invariant. Therefore, the isomorphism between $A_{<b}$ and $A_{<v}$ fixes this cut.  This means that the witnessing isomorphism of $A_{<b}\cong A_{<v}$ fixes $z$, as $z<x\in\{y|A_{<y}\cong A_{<x}\}$. This implies that the definition of the automorphism orbit of $z$ must be the same in the two structures, so $b$ satisfies the hypothesis of Lemma \ref{fixaut}. As $\SR(A_{<b})\leq\SR(A)\leq\gamma$, this gives that the automorphism orbit of $z$ is definable by a $\Sigma^{in}_{\gamma+4}$ formula, as desired.

Subcase 2: There is an $a\in A$ and there are cofinally many $x\in A_{\geq a}$ such that the set $\{y|[a,y)\cong[a,x)\}$ is bounded in $A_{\geq a}$.

Apply the argument in subcase 1 to $(A+B)_{\geq a}$.

Subcase 3: Otherwise.

Given a $z\in A$ we will find a $b\in A$ that satisfies the hypothesis of Lemma \ref{fixaut}. Let $a_0>z$ be a point such that $\{y|A_{<y}\cong A_{<a_0}\}$ is unbounded. We can find such an $a_0$ as we are not in subcase 1. Next, take $b>a_0$ such that $\{y|[a_0,y)\cong[a_0,b)\}$ is unbounded, and $A_{<b}\cong A_{<a_0}$. It is possible to find such a $b$ because we are not in subcase 2 so the set of points $c$ with $\{y|[a_0,y)\cong[a_0,c)\}$ unbounded is final in $A$, and the set of $d$ with $A_{<d}\cong A_{<a_0}$ is unbounded in $A$. We can then take $b$ in the necessarily non-empty intersection of a final and unbounded set.

We show that this $b$ has the desired properties within $A+B$. To see this, we must consider $v>b$ such that $(A+B)_{<v}\cong A_{<b}$. Because we are in case 2, we can assume that $v\in A$ and therefore $(A+B)_{<v}\cong A_{<v}$. Define $C=[a_0,b)$ and $D=[a_0,v)$. Observe that $C\sqsubseteq D$ by construction. Furthermore, as $\{y|[a_0,y)\cong [a_0,b)\}$ is unbounded in $A$, we have that there is a $y>v$ such that $D\sqsubseteq [a_0,y) \cong C$. In other words, $C$ and $D$ are both initial in each other.

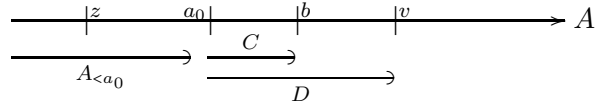
\[
\xymatrix@R=0.1pc{
\ar[rrrrrr] & |^z & {}^{a_0}| & |^{b} & |^v &   & A \\
\ar@{-)}_{A_{<a_0}}[rr]  & &   \ar@{-)}^{C}[r] & & & & \\
& &   \ar@{-)}_{D}[rr] & & & &
} \]

Name the isomorphism witnessing $A_{<b}\cong A_{<a_0}$, $\sigma$. Note that $A_{<v}\cong A_{<a_0}$ and name the witnessing isomorphism $\tau$. Say that $\sigma(a_0)\leq\tau(a_0)$. This gives that $C=[\sigma(a_0),a_0)\sqsupseteq [\tau(a_0),a_0)=D$. Now, Lemma \ref{lin} gives that $C=D$. If $\tau(a_0)\leq\sigma(a_0)$ the same argument gives $D\sqsupseteq C$ and we can also conclude $C=D$. This isomorphism along with fixing $A_{<a_0}$ provides an isomorphism between $A_{<b}$ and $A_{<v}$ that fixes $z$. This provides that the definition of the automorphism orbit of $z$ must be the same in the two structures, showing that $b$ satisfies the hypothesis of Lemma \ref{fixaut}. As $\SR(A_{<b})\leq\SR(A)\leq\gamma$, this gives that the automorphism orbit of $z$ is definable by a $\Sigma^{in}_{\gamma+4}$ formula, as desired.
\end{proof}

The following lemma is directly from \cite{MonIntermediate} and is proven in a similar way to the above result.

\begin{lemma}\label{seg}
(\cite[Lemma 4.7]{MonIntermediate}) If $L$ is a linear order such that $\forall x\in L ~ \SR(L_{\leq x})\leq\alpha$ then $ \SR(L)\leq\alpha+4$.
\end{lemma}

The following immediate corollary of this lemma allows us to find a low Scott rank bound for omega-sums of simple linear orders.

\begin{corollary}\label{omsum}
If $L$ can be written as an $\omega$ indexed sum of the form
$$L=\sum_{i\in \omega} A_i$$
where $\SR(A_i)\leq\gamma$, we have that $\SR(L)\leq\gamma+8$.
\end{corollary}

\begin{proof}
Let $a_i\in A_i$ be chosen arbitrarily. Consider 
$$L_k=\sum_{i\in k} A_i = \sum_{i\in k} A_{i,<a_i}+1+A_{i,>a_i}=A_{0,<a_0}+\big(\sum_{i < k-1} 1+A_{i,>a_i} + A_{i+1,<a_{i+1}}\big)+1+ A_{k,<a_{k}}.$$
Note that Lemma \ref{sum} gives that for each $i$ we have that $\SR(A_{i,>a_i} + A_{i+1,<a_{i+1}})\leq\gamma+4$. Therefore, Corollary \ref{finsum} gives that $\SR(L_k)\leq\gamma+4$. This gives that for any $x\in L$, for some $k$ $L_{<x}\sqsubseteq L_k$, and so $\SR(L_{<x})\leq\gamma+4.$ In total, Lemma \ref{seg} gives that $\SR(L)\leq\gamma+8$ as desired.

\end{proof}

With these notions established, we are now ready to move to the proof of $\om$-VC conjecture for linear orders.


\section{$\om$-VC for Linear Orders}
In this section we will show that all extensions of Linear Orders satisfy $\om$-VC. This will be achieved by using the criteria established in Lemma \ref{lem: VC splitting}. We will first define some relevant properties and tools needed for the proof. In particular, we will describe how we can change linear orders while maintaining the same $\alpha$ theory. From there the proof will split into cases. First we will consider the case that there are multiple points that form successive intervals of Scott rank at least $\alpha+\omega$. Then we will move to the case where all of the points are relatively c lose together.

\subsection{Understanding and justifying the major tools}

\subsubsection{The replacement lemmas}
In order to explore the space of linear orders that are $\alpha$ equivalent to a given linear order $L$, we will need ways to transform linear orders without changing their $\alpha$ theories. This will take the form of replacing intervals with $\alpha$ equivalent linear orders of certain Scott ranks. We show this always possible with finite error on the Scott rank, so long as we have that the structure is $(\alpha,\alpha+\omega)$-small.

\begin{lemma}\label{prime}
There is a non-decreasing function $f:\omega\to\omega$ which, given an $(\alpha,\alpha+\omega)$-small structure $L$ with $\SR(L)\geq\alpha+n$, guarantees that there is a structure $P$ with
$$L\equiv_{\alpha+n} P \text{  and   }  \alpha+n\leq\SR(P)\leq\alpha+f(n).$$
\end{lemma}

\begin{proof}
By the observation in Lemma \ref{bound} and Lemma \ref{intsr}, the statement that a structure is $\alpha+n$ equivalent to $L$ and has Scott rank greater than or equal to $\alpha+n$ is $\Pi_{\alpha+f(n)}$ for some function $f:\omega\to\omega$. Therefore, by the type omitting theorem for infinitary logic (see \cite{Part2} Chapter 2.4), there is a model $P$ such that $P\equiv_{\alpha+n} L$ and $\SR(P)\geq\alpha+n$ that omits all countably many of the non-$\Sigma_{\alpha+f(n)}^{in}$ supported $\Pi_{\alpha+f(n)}^{in}$ types that occur in models $\alpha$ equivalent to $L$ (there are only countably many by $(\alpha,\alpha+\omega)$-smallness). As the resulting structure is $\alpha$ equivalent to $L$, it has no non-$\Sigma_{\alpha+f(n)}^{in}$ supported $\Pi_{\alpha+f(n)}^{in}$ types. This gives that $\SR(P)\leq\alpha+f(n)$.
\end{proof}

In practice, this lemma is used to replace intervals in a linear order. This is important because, often times, replacing an interval will actually maintain the Scott rank of the overall structure so the resulting structure will be a witness to the key property that proves $\om$-VC.

\begin{lemma}\label{interval}
There is a non-decreasing function $f:\omega\to\omega$ which, given an $(\alpha,\alpha+\omega)$-small linear order, $L$ and any $x,y\in L$ such that $\SR([x,y])\geq\alpha+n$, guarantees that there is a linear order $P$ with
$$[x,y]\equiv_{\alpha+n} P \text{  and   }  \alpha+n\leq\SR(P)\leq\alpha+f(n).$$
\end{lemma}

\begin{proof}
Because of the previous lemma it is enough to show that each interval of $L$ is also $(\alpha,\alpha+\omega)$-small. If there were some $[x,y]$ that were not $(\alpha,\alpha+\omega)$-small, for some $m$ there would be continuum many $\alpha+m$ types realized among models $\alpha$ equivalent to $[x,y]$. Index these types $p_A$ by reals $A\in2^\omega$ and say that the $A$th type is realized by $z_A$ in $M_A\equiv_\alpha[x,y]$. Note that
$$L\equiv_\alpha L_{<x}+M_A+L_{>y},$$
by \cite[Remark 1.5.4]{Ste78}.
Furthermore, it is not difficult to see that $(x,z_A,y)$ realizes a different $\alpha+m$ type for each $A\in2^\omega$ where $x$ and $y$ are the maximal and minimal elements of $M_A$ respectively. Therefore, there are continuum many $\alpha+m$ types realized among models $\alpha$ equivalent to $L$, a contradiction to smallness.
\end{proof}

It is worth nothing that the above proof works equally well if $x=-\infty$ or if $y=\infty$. In other words, initial and end segments are also $(\alpha,\alpha+\omega)$-small. Similarly, using open versus closed intervals makes little difference for the argument.

Sometimes, we want to replace convex suborders that are not bounded by explicit elements in the order. In this case we must make some minor concessions on the Scott rank of the replacement order, but ultimately a similar result holds. For the following we assume we are given the fixed function $f$ from Lemma \ref{interval}.

\begin{lemma}\label{sub}
If $L$ is an $(\alpha,\alpha+\omega)$-small linear order, then for any convex suborder $K\subset L$ there is an order $K'\equiv_\alpha K$ such that $\SR(K')\leq\alpha+f(0)+8$.
\end{lemma}

\begin{proof}
Pick some $x\in K$. We first show that for $K_{\geq x}$, there is a $A\equiv_\alpha K_{\geq x}$ such that $\SR(A)\leq\alpha+f(0)+8$.  If $K$ has a maximal element, $y$, $K_{\geq x}=[x,y]$, so the result follows directly from Lemma \ref{interval}. Thus, we may assume that $K$  has no maximal element. This means that, we can find a cofinal, order preserving, injective map $i:\omega\to K$ with $i(0)=x$ by Lemma \ref{cofseq}. In this case, $K_{\geq x}=\sum_\omega[i(n),i(n+1))$. Replace each $(i(n),i(n+1))$ of Scott rank greater than $\alpha$ with an $\alpha$-equivalent model with Scott rank less than $\alpha+f(0)$. Call the resulting order $C$. Note that Corollary \ref{omsum} yields that $\SR(C)\leq\alpha+f(0)+8$. Doing the same procedure for $K^*_{\geq x}$ yields that we can replace $K_{\leq x}$ with $B$ such that $\SR(D)\leq\alpha+f(0)+8$. In total, this gives that $K\equiv_\alpha C+1+D=K'$ and that $\SR(K')\leq \alpha+8$ by Corollary \ref{finsum}, as desired.
\end{proof}

\subsubsection{The convex $\alpha$ equivalence relation} 
We consider a convex equivalence relation on linear orders in a manner similar to (but not the same as) \cite{Ste78} and \cite{MonIntermediate} . This will be key to the process of the proof. In particular, we will distinguish between the case that $L/\sim_{\alpha+\omega}$ has less than 3 elements and the case that is has at least 3 elements, and distinguish between the case that $L/\sim_{\alpha+f(0)+10}$ is an ordinal and the case where it is not an ordinal.

\begin{definition}
Given an ordinal $\alpha\in\omega_1$, a linear order $L$, and two points $a<b\in L$  we say that $$a\sim_\alpha b$$ if and only if $\SR((a,b))<\alpha$.
\end{definition}

Note that for any ordinal $\alpha$, it follows immediately from Corollary \ref{finsum} that $\sim_\alpha$ is an equivalence relation. This stands in contrast to previous versions of this equivalence relation which would only work for limit ordinals $\alpha$. This is because they were using unparameterized Scott rank or an older notion of Scott rank, so they did not have Lemma \ref{leasteq}. This is a tangible benefit of using the notion of robust Scott rank defined earlier in this paper. Even if this relation has slightly different properties from similar relations already in the literature, because it is defined analogously, we will use the standard notation for it and related concepts. For example, $[x]_\alpha$ for $x\in L$ refers to the set of points in $L$ that are in the $\alpha$ equivalence class of $x$ within $L$. It is worth noting that as these equivalence classes are convex subsets of a linear order, they inherit the structure of a linear order themselves.

\subsubsection{Upwards closedness}
In an effort to formalize exactly what part of a linear order deserves our focus, we turn to a ranked form of having large end segments. 

\begin{definition}
A linear order, $L$, is $\gamma$ upwards-closed (or $\gamma-UC$) if for all $x\in L$ we have that $\SR(L_{\geq x})\geq\gamma$. 
\end{definition}

In the analysis that follows we will often be concerned with linear orders that are $(\alpha+\omega)-UC$.  We will refer to this as condition \textbf{(1)} for the sake of parsimony. This notion may seem strange at first; a motivating example will elucidate the usefulness of this idea. Consider the ordinal $\omega^\beta$ for any ordinal $\beta$. Note that any end segment is isomorphic to the whole of $\omega^\beta$. In particular, $\SR(\omega^\beta)=\SR(\omega^\beta_{\geq x})= 2\beta$ so $\omega^\beta$ is $(2\beta)-UC$. Therefore, the complexity of the order is really in its end segments and any initial segment can be replaced without lowering the complexity of the overall structure too much.

\subsection{Case 1: $L/\sim_{\alpha+\omega}$ has many elements}

In this section we prove the theorem in the case that $L/\sim_{\alpha+\omega}$ has at least 3 elements. The formula that splits the model on states that there is an initial or end segment of Scott rank between $\alpha+f(0)+9$ and $\alpha+f(f(0)+9)$. To be precise we find models $\alpha$ equivalent to an $(\alpha,\alpha+\omega)$-small $L$ with Scott rank at least $\alpha+\omega$ that disagree on one of the following formulae:
$$\psi_\geq:~~~~ \exists x ~\bigvee_{f(0)+9\leq i\leq f(f(0)+9)}\SR(L_{\geq x})=\alpha+i,$$
$$\psi_\leq:~~~~ \exists x ~\bigvee_{f(0)+9\leq i\leq f(f(0)+9)}\SR(L_{\leq x})=\alpha+i.$$
Because of Lemma \ref{intsr}, $\psi_\geq$ and $\psi_\leq$ are $\Sigma_{<(\alpha+\omega)}^{in}$. In the next two subsections we will explain how to satisfy these formula in a model $\alpha$-equivalent to $L$ and then we will explain how to satisfy one of their negations in a model $\alpha$-equivalent to $L$. This will imply condition \ref{part: VC splitting 2} of Lemma \ref{lem: VC splitting} in this case.

\subsubsection{Satisfying $\psi_\leq$ and $\psi_\geq$.}

In this section we will show that any $(\alpha,\alpha+\omega)$-small linear order $L$ with $|L/\sim_{\alpha+\omega}|\geq 3$ can be made to satisfy $\psi_\leq$ and $\psi_\geq$ while maintaining the $\alpha$ theory. 

\begin{lemma} \label{satisfiespsi}
If $L$ is $(\alpha,\alpha+\omega)$-small and $L/\sim_{\alpha+\omega}$ has 3 or more equivalence classes, there is an $A\equiv_\alpha L$ such that $A\models\psi_\leq$ and $\SR(A)\geq\alpha+\omega$ and a $B\equiv_\alpha L$ such that $B\models\psi_\geq$ and $\SR(B)\geq\alpha+\omega$.
\end{lemma}

\begin{proof}
 Let $x<y<z$ be elements of $L$ in different $\sim_{\alpha+\omega}$ equivalence classes. Note that both $\SR(L_{\leq y})\geq\alpha+\omega$ and $\SR(L_{\geq y})\geq\alpha+\omega$. This is because they both have $\Delta_0^{in}$ definable subsets of Scott rank at least $\alpha+\omega$ ($[x,y]$ and $[y,z]$ respectively). Therefore, by Lemma \ref{interval} we can find $M\equiv_\alpha L_{\leq y}$ with $\alpha+f(0)+9\leq\SR(M)\leq\alpha+f(f(0)+9)$ and $N\equiv_\alpha L_{\geq y}$ with $\alpha+f(0)+9\leq\SR(N)\leq\alpha+f(f(0)+9)$. We can now observe that
$$M+L_{> y}\models \psi_\leq,$$
$$L_{< y} + N\models \psi_\geq.$$
Furthermore, both of these models are $\alpha$ equivalent to $L$ and are Scott Rank at least $\alpha+\omega$. Therefore, these models are the required $A$ and $B$ respectively.
\end{proof}

With this lemma we are able to satisfy $\psi_\leq$ and $\psi_\geq$, now all that remains is showing that we can satisfy one of their negations.

\subsubsection{Satisfying $\lnot\psi_\leq$ or $\lnot\psi_\geq$.}

We now aim to find an $\alpha$ equivalent model to $L$ satisfying $\lnot\psi_\leq$ or $\lnot\psi_\geq$. Unlike the first case, we will not be able to guarantee a particular one of these formulae. We proceed by defining a normal form for linear orders that we will always be able to transform our order into while maintaining the $\alpha$ theory. From this, our desired result will follow immediately. This will be useful more generally for the proof moving forward. For this reason, we also introduce a second important condition before moving on with the proof. This condition will only be used in later sections.

\begin{definition}
We will say that a linear order, $L$, satisfies condition \textbf{(2)} if for all $x\in L$ we have that $\SR(L_{\leq x})<\alpha+\omega$.
\end{definition}

Condition \textbf{(2)}, like condition \textbf{(1)}, is inspired by the analysis of ordinals. In particular, $\omega^{\alpha+\omega}$ has condition \textbf{(2)}, but no larger ordinals do. It also has condition \textbf{(1)} unlike any smaller ordinals. In this, condition \textbf{(1)} and condition \textbf{(2)} together pick out orders that act like $\omega^{\alpha+\omega}$  in an abstract sense. This is naturally useful as we are concerned with orders of Scott rank $\alpha+\omega$, and this is the canonical example of such an order.

One may wonder if there are very many other examples of orders that satisfy condition \textbf{(1)} and condition \textbf{(2)}. As it turns out they are quite common. In fact,  linear orders of sufficiently large Scott rank are often $\alpha$ equivalent to a linear order that contains a suborder satisfying condition \textbf{(1)} and condition \textbf{(2)}. The following lemma formalizes this idea by demonstrating the existence of the aforementioned normal form.

\begin{lemma}\label{finding12}
Consider an $(\alpha,\alpha+\omega)$-small linear order with $\SR(L)\geq\alpha+\omega$. Either $L\equiv_\alpha K+P$ or  $L^*\equiv_\alpha K+P$ where $K$ has \textbf{(1)} and \textbf{(2)} and $\SR(P)\leq\alpha+f(0)+8$.
\end{lemma}

\begin{proof}
We break the proof into cases based on the behavior of $L/\sim_{\alpha+\omega}$. We consider the case that it is a point, the case that it is not dense but is larger than a point, and the case that it is dense.

First consider the case where $L/\sim_{\alpha+\omega}$ is a point. Say there is a point $x\in L$ such that $\SR(L_{\geq x})\geq \alpha+\omega$ (if there is not, there is a point with $\SR(L_{\leq x})\geq \alpha+\omega$ and we consider $L^*$ instead). By Lemma \ref{interval}, we can replace $L_{<x}$ with a model of Scott rank at most $\alpha+f(0)$. 

Consider $y>x$. By the fact that $L/\sim_{\alpha+\omega}$ is a point, $\SR([x,y])<\alpha+\omega$, which yields that $\SR(L_{\geq y})\geq \alpha+\omega$ by Lemma \ref{leasteq}. Furthermore, Lemma \ref{leasteq} guarantees that $\SR(L_{\leq y})<\alpha+\omega$ as well. 

For $y<x$, $\SR(L_{\geq y})\geq \SR(L_{\geq x})\geq \alpha+\omega$. Furthermore, $SR(L_{\leq y})\leq \SR(L_{\leq x})< \alpha+\omega$. Therefore, every element $z\in L$ has $\SR(L_{\geq z})\geq \alpha+\omega$ and $\SR(L_{\leq z})< \alpha+\omega$, as desired.

Now consider the case where  $L/\sim_{\alpha+\omega}$ is more than one point and not dense. By non-density we know there are two equivalence classes that are adjacent to each other. Let those classes be given by $[x]_{\alpha+\omega}$ and $[y]_{\alpha+\omega}$ for some $x,y\in L$. As $x$ and $y$ are in different classes note that $\SR([x,y])\geq \alpha+\omega$. We can split the order $[x,y]$ between the elements in $x$'s equivalence class and $y$'s equivalence class to realize this order as a sum. To be more precise, we may note that
$$[x,y]=[x,y]\cap[x]_{\alpha+\omega}+[x,y]\cap[y]_{\alpha+\omega}.$$
Lemma \ref{sum} gives that one of these summands has Scott rank at least $\alpha+\omega$. Without loss of generality we assume it is $[x,y]\cap[x]_{\alpha+\omega}$ (or else we consider $L^*$).

Note that $L$ can be written as 
$$L=L_{\leq x} + [x,y]\cap[x]_{\alpha+\omega} + M,$$
for some linear order $M$. Let $Q\equiv_\alpha L_{\leq x}$ have $\SR(Q)\leq \alpha+f(0)$ and $P\equiv_\alpha M$ have $\SR(P)\leq \alpha+f(0)+8$ as in Lemmas \ref{interval} and \ref{sub}. We claim that
$$L'=Q+[x,y]\cap[x]_{\alpha+\omega} + P,$$
satisfies the desired properties.

To show this, we demonstrate that $Q+[x,y]\cap[x]_{\alpha+\omega}$ has \textbf{(1)} and \textbf{(2)}. 

We start with \textbf{(1)}. It is clearly enough to show that $[x,y]\cap[x]_{\alpha+\omega}$ is $(\alpha+\omega)$-UC. However, this follows as if $\SR(([x,y]\cap[x]_{\alpha+\omega})_{\geq z})<\alpha+\omega$, because $\SR([x,z])<\alpha+\omega$, Lemma \ref{leasteq} gives that $\SR([x,y]\cap[x]_{\alpha+\omega})<\alpha+\omega$. This is a contradiction to our choice of $[x,y]\cap[x]_{\alpha+\omega}$.

To see \textbf{(2)}, by Lemma \ref{leasteq} it is enough to show that $[x,y]\cap[x]_{\alpha+\omega}$ has \textbf{(2)}. However, this is immediate as any $z\in[x]_{\alpha+\omega}$ has $\SR([x,z])<\alpha+\omega$ by definition. Therefore, $L'$ is of the desired form.

Finally we consider the case where $L/\sim_{\alpha+\omega}$ is dense. In this case we consider an increasing sequence of points $x_i$ in separate $\sim_{\alpha+\omega}$ equivalence classes. We can write $L=M+U$ where $U$ is the set of points greater than all of the $x_i$ and $M$ is the set of points less than some $x_i$. By Lemma \ref{sub} there is a $P\equiv_\alpha U$ such that  $\SR(P)\leq \alpha+f(0)+8$. Furthermore, by Lemma \ref{interval}, for each $[x_i,x_{i+1}]$ there is a $K_i$ such that $\alpha+i\leq\SR(K_i)<\alpha+\omega$ and $K_i\equiv_\alpha[x_i,x_{i+1}]$ (by convention we let $x_0=-\infty$ here). Note that this gives
$$L\equiv_\alpha \big(\sum_{i\in\omega} K_i\big)+P.$$
Using Lemma \ref{finsum}, it is not difficult to see that $K=\sum_{i\in\omega} K_i$ has properties \textbf{(1)} and \textbf{(2)} as desired. 
\end{proof}

We can now observe that a $K+P$ in the form of the above theorem satisfies $\lnot\psi_\geq$ and $(K+P)^*$ satisfies $\lnot\psi_\leq$. In other words, the following is apparent.

\begin{corollary}\label{notpsi}
If $L$ is an $(\alpha,\alpha+\omega)$-small linear order with $SR(L)\geq\alpha+\omega$, then there is an $A$ such that $\SR(A)\geq\alpha+\omega$, $A\equiv_{\alpha} L$ and either $A\models\lnot\psi_\leq$ or $A\models\lnot\psi_\geq$.
\end{corollary}

This combined with Lemma \ref{satisfiespsi} immediately gives our goal lemma for the section.

\begin{lemma}\label{3final}
If $L$ is an $(\alpha,\alpha+\omega)$-small linear order and $L/\sim_{\alpha+\omega}$ has 3 or more equivalence classes, there are $A$ and $B$ such that $A\equiv_\alpha B\equiv_\alpha L$ with $\SR(A),\SR(B)\geq\alpha+\omega$, yet $A\not\equiv_{\alpha+\omega}B$.
\end{lemma}

\subsection{Case 2: $L/\sim_{\alpha+\omega}$ has few elements}

The case where there are less than $3$ $\sim_{\alpha+\omega}$ equivalence classes is the critical case of the theorem. The analysis is trickier and has more subcases. Luckily, the results of the previous sections allow us to consider only orders with very specific properties. To start with, we observe that Lemma \ref{finding12} applies. In particular, without loss of generality we will write $L=K+P$ with $K$ having conditions \textbf{(1)} and \textbf{(2)}. The first case we will approach is the case that taking the quotient by $\sim_{\alpha+f(f(0)+9)+1}$ is not an ordinal. The remaining case is technical and specific, but ultimately doable. We will need to introduce one more condition to consider and then conduct an analysis of initial segments of the order. Then we will be able to prove the theorem.

\subsubsection{$L/\sim_{\alpha+f(f(0)+9)+1}$ is not an ordinal}

Let us prove the theorem in the case that we have non-ordinal behavior at the level of $\alpha+f(f(0)+9)+1$. Recall that we are assuming that $L=K+P$ as in Lemma \ref{finding12}.

\begin{lemma}\label{nonOrdPlus}
Let $L$ be $(\alpha,\alpha+\omega)$-small and such that $L/\sim_{\alpha+f(f(0)+9)+1}$ is not an ordinal. Then there are $A$ and $B$ such that  $A\equiv_\alpha B\equiv_\alpha L$, $\SR(A),\SR(B)\geq\alpha+\omega$ yet $A\not\equiv_{\alpha+\omega}B$. 
\end{lemma}

\begin{proof}
We are able to split on the formula $\psi_\leq$ introduced in the previous section in this case. We demonstrate first how to find a model that satisfies $\psi_\leq$ and then we show how to satisfy its negation.

We now find a suitable model that satisfies $\psi_\leq$. By Lemma \ref{seg} there must be some $x\in K$ such that $\SR(K_{< x})\geq \alpha+f(0)+9$. Using Lemma \ref{interval} there is a $C$ such that $C\equiv_\alpha K_{<x}$ with $f(0)+9\leq\SR(C)\leq\alpha+f(f(0)+9)$. Let $B:=C+L_{\geq x}$. It is clear that $B\equiv_\alpha L$ and that $\SR(B)\geq\alpha+\omega$ by condition \textbf{(1)}. Note that $x$ is a witness to $\psi_\leq$, so $B\models \psi_\leq$. 

We now find a suitable model that satisfies $\lnot\psi_\leq$. By assumption, we can find a decreasing sequence of points $a_i\in L$ such that $\SR((a_{i+1},a_i))\geq\alpha+f(f(0)+9)+1$. Moreover, if we let $S$ be the (possibly empty) set of elements in $L$ smaller than all of the $a_i$ we know from Lemma \ref{sub} that there is a $Q$ such that $S\equiv_\alpha Q$ and $\SR(Q)\leq\alpha+f(0)+8$. We can now define the order $A$ as follows:
$$A:=Q+\Big(\sum_{i\in\omega^*}(a_{i+1},a_i)+1\Big)+L_{>a_0}.$$
Note that $A\equiv_\alpha L$ by construction. For the sake of parsimony, we can also define $M$ as follows:
$$M:=\sum_{i\in\omega^*}(a_{i+1},a_i)+1,$$
so that $A=Q+M+L_{>a_0}.$

We show that $A\models\lnot\psi_\leq$ by considering a possible witness to the formula. Take $x\in A$. If $x\in Q$ we have that 
$$\SR(A_{\leq x})=\SR(Q_{\leq x})\leq\SR(Q)\leq\alpha+f(0)+8,$$
so $x$ is not a witness to $\psi$. Instead consider $x\in M+L_{>a_0}.$ By the construction of $M$ there is some $i$ such that $x$ is greater than all of the elements of $(a_{i+1},a_i)$. Given this, we see that
$$\SR(A_{\leq x})\geq\SR((a_{i+1},a_i))\geq\alpha+f(f(0)+9)+1,$$
as $(a_{i+1},a_i)$ is an interval within $A_{\leq x}$. Thus, $x$ is not a witness to $\psi$ in this case either. As these cases are exhaustive, there are no witnesses to $\psi_\leq$ in $A$, so $A\models\lnot\psi_\leq$ as desired.

Finally, we note that $\SR(A)\geq\alpha+\omega$ as it contains $L_{>a_0}$ as an end segment which necessarily has Scott rank $\alpha+\omega$ by condition \textbf{(1)} for $K$, noting that $L_{>a_0}=K_{>a_0}+P$. This means that $A$ and $B$ have the required conditions, so we have proved the lemma.
\end{proof}

Therefore, moving forward we can always assume that $L/\sim_{\alpha+f(f(0)+9)+1}$ is an ordinal.

\subsubsection{A third condition}

We now introduce a third important condition exhibited by some linear orders. This condition formalizes the idea that no matter what replacement you do, the order remains relatively small. 

\begin{definition}
We will say that a linear order, $L$, satisfies condition \textbf{(3)} if for all $x\in L$ we have that $L_{\leq x}\not\equiv_\alpha L+M$ for any linear order $M$.
\end{definition}

Unlike condition \textbf{(1)} and condition \textbf{(2)}, condition \textbf{(3)} is not something seen in ordinal examples. We desire to assume condition \textbf{(3)} for $K$ where $L=K+P$ as we move forward. While technical, it is needed to dispose of the last cases explored in this proof. In this section we consider the case that $K$ does not have condition  \textbf{(3)} so we are able to assume it in the future cases.

\begin{lemma}\label{noCon3}
Let $L$ be $(\alpha,\alpha+\omega)$-small and such that $L=K+P$ where $K$ has \textbf{(1)} but not \textbf{(3)}. Then there are $A$ and $B$ such that  $A\equiv_\alpha B\equiv_\alpha L$, $\SR(A),\SR(B)\geq\alpha+\omega$ yet $A\not\equiv_{\alpha+\omega}B$. 
\end{lemma}

\begin{proof}
Let $x\in K$ witness that $K_{\leq x}\equiv_{\alpha} K+M$. Now note that
$$L=K+P = K_{\leq x}+ K_{>x}+P\equiv_\alpha K+M+K_{>x}+P\equiv_\alpha K+M+K_{>x}+M+K_{>x}+P.$$
Because $K$ has \textbf{(1)} so does $K_{>x}$ as all final segments of $K_{>x}$ are also final segments of $K$. In particular, if $w$ comes from the first $K$, $y$ from the first $K_{>x}$ and $z$ from the second $K_{>x}$ we see that $\SR([w,y])\geq\alpha+\omega$ and $\SR([y,z])\geq\alpha+\omega$. In other words, $(K+M+K_{>x}+M+K_{>x}+P)/\sim_{\alpha+\omega}$ has at least 3 equivalence classes. Therefore, we can apply Lemma \ref{3final} to obtain the result.
\end{proof}

We can now move forward assuming condition \textbf{(3)}.

\subsubsection{The final case}

We now prove the theorem in the case that $L=K+P$ where $L$ is $(\alpha,\alpha+\omega)$-small, $K$ satisfies \textbf{(1)}, \textbf{(2)} and \textbf{(3)} and such that $L/\sim_{\alpha+f(0)+10}$ is not an ordinal. There are a couple steps here, so this will be broken up into several subsections for organization. In particular, we will first need to understand the behavior of linear orders who have quotients that are ordinals a bit better. Then we will consider cases based on what sorts of orders are $\alpha$ equivalent to the initial segments of $K$.

\subsubsection{The comparability lemma}
In this subsection we aim to show that linear orders with limit ordinal quotients behave like ordinals in the sense that $M\sqsubseteq N$ (recall this means that $M$ is initial in $N$) if and only if each initial segment of $M$ is initial in $N$. Note that this is not true in general. For example, $\mathbb{Z}$ is not initial in $\omega^*$ despite the fact that all of its initial segments are. In order to show such a result, we must make the conversion to ordinals even more carefully than before. This motivates the following definition.

\begin{definition}
Given a linear order $K$ and a $k\in\omega$, call 
$$\text{Iso}(\alpha+k,L) = \{[x]_{\sim_{\alpha+k}}|x\in K\}/\cong.$$
 Because this is a countable set, we can fix an bijective coloring $c:\text{Iso}(\alpha+k,K)\to \omega$. We define the structure $C(K)$ as a model of the theory of linear orders with countable colors $(K/\sim_{\alpha+k},<,\{P_i\}_{i\in\omega})$, where $[x]\in P_i$ if and only if $c([x])=i$. 
\end{definition}

In particular, the colored ordinal $C(K)$ does not lose information about the original linear order in the way that $K/\sim_{\alpha+k}$ does. This will allow us to use embeddings of colored ordinals to construct embeddings of the original orders in question. With this in mind, for the sake of parsimony we write $K\sqsubseteq_cN$ if there is a colored-order-initial-embedding from $K$ to $N$. We can now prove the desired result.

\begin{lemma}\label{compare}
Consider linear orders $M$ and $N$ such that $M/\sim_{\alpha+k}$ is a limit ordinal for some $k\in\omega$. $M\sqsubseteq N$ if and only if $M_{\leq x}\sqsubseteq N$ for all $x\in M$.
\end{lemma}

\begin{proof}
Assume that if $M_{\leq x}\sqsubseteq N$ for all $x\in M$. We aim to show that $C(M)_{\leq [x]}\sqsubseteq_c C(N)$ for all $[x]\in C(M)$ and then use this to show that $M\sqsubseteq N$.

Fix some $[x]\in C(M)$. By assumption $C(M)$ does not have a largest element. Let $[z]>[x]$. Given an element of $[y]\in C(M)$, by abuse, we let $[y]$ also denote the linear order that corresponds to the color of $[y]$ in $C(M)$. With this in mind, we can note that
$$M_{\leq z} = \sum_{[y]<[z]}[y] + [z]_{\leq z}$$
by construction. By assumption, $M_{\leq z}\sqsubseteq N$, and therefore
$$ \sum_{[y]<[z]}[y] + [z]_{\leq z}\sqsubseteq N.$$
Given two elements $a,b$ in this initial segment the interval between them in $N$ is exactly the same as the interval between them in $M$. Therefore, if $a,b\in[y]$ they still have the property that $\SR((a,b))<\alpha+k$ and if $a\in[y]$ and $b\in[y']$ for $[y]\neq [y']$ they still have the property that $\SR((a,b))\geq\alpha+k$. Therefore, 
$$C\big( \sum_{[y]<[z]}[y]\big)=C(M)_{<[z]},$$
and $C(M)_{<[z]}$ must be initial in $C(N)$. As $[z]>[x]$, we have that $C(M)_{\leq [x]}\sqsubseteq_c C(N)$.

Because we showed this for any $[x]$, all initial segments of $C(M)$ are initial in $C(N).$ Because $C(M)$ is a colored ordinal, by explicit recursive construction we have that $C(M)\sqsubseteq_c C(N)$. However, it immediately follows that
$$M= \sum_{[y]\in C(M)}[y] \sqsubseteq \sum_{[y']\in C(N)}[y']=N,$$
giving the desired result.
\end{proof}

\subsubsection{The non-$\alpha$ closed case}
This section we consider the case that $L$ is $(\alpha,\alpha+\omega)$-small, $L/\sim_{\alpha+f(0)+10}$ is an ordinal and that $L=K+P$ with $K$ satisfying \textbf{(1)},\textbf{(2)} and \textbf{(3)}; this is the case that is remaining given the work done above. That being said, we must break this proof for this into two sections depending on if the initial segments of $K$ are closed under $\alpha$ equivalence (i.e. if all initial segment are only $\alpha$ equivalent to orders isomorphic to another initial segment). We handle the case where the initial segments of $L$ are not closed under $\alpha$ equivalence first. 

\begin{lemma}\label{nonClosed}
If $K$ is a linear order such that 
\begin{itemize}
	\item $K$ satisfies \textbf{(1)}, \textbf{(2)} and \textbf{(3)},
	\item there is some $x\in K$  and linear order $N$ such that, $K_{\leq x}\equiv_\alpha N$ and for all $y\in K$, $N\not\cong K_{\leq y}$,
\end{itemize}
and $L=K+P$, then there are $A$ and $B$ such that $\SR(A),\SR(B)\geq\alpha+\omega$, and $A\equiv_\alpha B\equiv_\alpha L$ yet $A\not\equiv_{\alpha+\omega}B$. 
\end{lemma}

\begin{proof}
Consider $K_N:=N+K_{>x}$ where $N$ is the order assumed to exist above, and let $L_N=K_N+P$.

We first consider the case where $K_N$ is missing some initial segment of $K$, say $K_{\leq z}$. Then, formally speaking
$$K_N\models \lnot\exists w ~ K_{N,\leq w}\cong K_{\leq z}.$$ 
Note also that 
$$L_N\models \lnot\exists w ~ K_{N,\leq w}\cong K_{\leq z},$$ 
as possible witnesses $w\in P$ would have $\SR(L_{\leq w})\geq\alpha+\omega$ disqualifying the possibility of isomorphism with $K_{\leq z}$ which has strictly lower Scott rank that $\alpha+\omega$ by condition \textbf{(2)}.
Call this formula $\psi$. By condition \textbf{(2)} on $K$ and Lemma \ref{intsr}, $\psi$ has complexity less than $\alpha+\omega$. Note that $L\models\lnot\psi$ as $z$ itself is a witness to the property. Therefore, we can take $A=L$ and $B=L_N$ to complete the proof in this case. 

We now only need to consider the case where all initial segments of $K$ are initial in $K_N$. In other words, for each $z\in K$ there is a $y\in K_N$ such that $K_{\leq z}\cong K_{N,\leq y}$. Note that because $N$ is not initial in $K$ it is not initial in any of its initial segments. Thus, for each $z$, its corresponding $y\in K_N=N+K_{>x}$ cannot come from the $K_{>x}$ summand, or else $K_{N,\leq y}\cong K_{\leq z}$ has $N$ as an initial segment. Therefore, all initial segments of $K$ must actually be initial in $N$. However, by Lemma \ref{compare} this means that $K\sqsubseteq N$. This is a contradiction to condition \textbf{(3)}. Thus, this case cannot occur and we have completed the proof.
\end{proof}

\subsubsection{The $\alpha$ closed case}
In this section we complete the proof in the case that $K$ satisfies \textbf{(1)},\textbf{(2)} and \textbf{(3)}. The only case that remains is very specific. In particular, it is the case that the initial segments of $K$ are closed under $\alpha$ equivalence (i.e. for each $x\in K$ and $N\equiv_\alpha K_{\leq x}$ for some $y\in K$ we have that $N\cong K_{\leq y}$) and that $L/\sim_{\alpha+f(f(0)+9)+1}$ is an ordinal. We will see that this is, in fact, impossible. 

\begin{lemma}\label{Closed}
If $K$ satisfies \textbf{(1)} and \textbf{(3)} and $K/\sim_{\alpha+f(f(0)+9)+1}$ is an ordinal, then it is not possible that all models $\alpha$-equivalent to some $K_{\leq x}$ for some $x\in K$ are isomorphic to $K_{\leq y}$ for some $y\in K$.
\end{lemma}

\begin{proof}
For the sake of contradiction assume that all models $\alpha$-equivalent to some $K_{\leq x}$ for some $x\in K$ are isomorphic to $K_{\leq y}$ for some $y\in K$. Note that $K$ has the property that each of its initial segments is $\alpha$-equivalent to one of its initial segments (namely itself). More formally, note that
$$K\models  \forall x\bigvee_{y\in K} ~ K_{\leq x} \equiv_\alpha K_{\leq y}.$$
Call this formula $\psi$. By Lemma \ref{bound} this formula is $\Pii_{\alpha+n}$ for some $n\in\omega$. Moreover, note that as $K$ satisfies \textbf{(1)} we have that
$$K\models \forall x\exists y ~ \SR((x,y))\geq \alpha+f(f(0)+9)+1.$$
Call this formula $\chi$. By Lemma \ref{intsr} this formula is $\Pii_{\alpha+m}$ for some $m\in\omega$. 

As $\SR(K)\geq\alpha+\omega$, there is some $D\not\cong K$ such that $K\equiv_{\alpha+n+m} D$. Note that
$$D\models\psi\land\chi.$$
Because of $\psi$ and the initial assumption, this means that all initial segments of $D$ are initial in $K$. Because of this, for any $x\in D$, $D_{\leq x}/\sim_{\alpha+f(f(0)+9)+1}$ must be an ordinal as it is initial in an ordinal. Thus, $D/\sim_{\alpha+f(f(0)+9)+1}$ is also an ordinal as all of its initial segments are ordinals. Furthermore, because of $\chi$, we have that $D/\sim_{\alpha+f(f(0)+9)+1}$  has no greatest element, so it must be a limit ordinal. Therefore, by Lemma \ref{compare} and $\psi$, $D\sqsubseteq K$. As they are non isomorphic, in fact, $D\sqsubset K$. Therefore there is some bound $b$ for $D$ in $K$. Now note that for some $M$,
$$K_{\leq b}=D+M\equiv_\alpha K+M,$$
A contradiction to condition \textbf{(3)}.
\end{proof} 

With this we have proven the following Lemma that simply puts this together this with Lemma \ref{nonClosed}.

\begin{lemma}\label{123}
If $K$ satisfies \textbf{(1)}, \textbf{(2)}  and \textbf{(3)} and $L=K+P$ is $(\alpha,\alpha+\omega)$-small, then there are $A$ and $B$ such that $A\equiv_\alpha B\equiv_\alpha L$, yet $A\not\equiv_{\alpha+\omega}B$. 
\end{lemma}

Putting this all together with Lemma \ref{nonClosed}, Lemma \ref{noCon3}, Lemma \ref{nonOrdPlus} and Lemma \ref{3final} allows us to conclude our goal.

\begin{theorem}
For any $(\alpha,\alpha+\omega)$-small linear order $L$ with $\SR(L)\geq\alpha+\omega$, there are $A$ and $B$ such that $A\equiv_\alpha B\equiv_\alpha L$, yet $A\not\equiv_{\alpha+\omega}B$. Therefore, the theory of linear orders satisfies $\om$-VC.
\end{theorem}

\end{document}